\documentclass[12pt]{article}
\usepackage[utf8]{inputenc}
\usepackage[T1]{fontenc}
\usepackage{lmodern}
\usepackage{amsmath,amsthm}
\usepackage{amsfonts}
\usepackage{amsmath}
\usepackage{amssymb}
\usepackage{colonequals}
\usepackage{tikz}
\usetikzlibrary{shapes}
\usepackage{mathtools,microtype}
\usepackage{hyperref}

\newcommand{\GG}{{\cal G}}
\renewcommand{\AA}{{\cal A}}

\newcommand{\PP}{{\cal P}}

\DeclareMathOperator{\Prb}{\mathbf{Pr}}
\DeclareMathOperator{\Ex}{\mathbf{E}}
\DeclareMathOperator{\tw}{\text{tw}}
\DeclareMathOperator{\td}{\text{td}}

\newtheorem{theorem}{Theorem}
\newtheorem{corollary}[theorem]{Corollary}
\newtheorem{lemma}[theorem]{Lemma}

\newtheorem{observation}[theorem]{Observation}

\theoremstyle{definition}

\newcommand{\des}[3]{#1_{#2\downarrow #3}}
\newcommand{\mal}[3]{#1_{#2\le#3}}

\newcommand{\unit}{30pt}

\title{On fractional fragility rates of graph classes\thanks{This work falls within the scope of L.I.A. STRUCO.}}
\author{Zden\v{e}k Dvo\v{r}\'ak\thanks{Charles University, Prague, Czech Republic.
E-mail: {\tt rakdver@iuuk.mff.cuni.cz}.
Supported by the Neuron Foundation for Support of Science under Neuron Impuls programme.}
\and Jean-S\'ebastien Sereni\thanks{Centre National de la Recherche Scientifique, CSTB (ICube), Strasbourg, France.
E-mail: \texttt{sereni@kam.mff.cuni.cz}. This work was partially supported by P.H.C. Barrande~40625WH.}}
\date{}

\begin{document}
\maketitle

\begin{abstract}
    We consider, for every positive integer~$a$, probability distributions on
    subsets of vertices of a graph with the property that every
    vertex belongs to the random set sampled from this distribution with probability at most~$1/a$.
    Among other results, we prove that for every positive integer~$a$ and every planar graph~$G$, there
    exists such a probability distribution with the additional property that
    deleting the random set creates a graph with component-size at most~$(\Delta(G)-1)^{a+O(\sqrt{a})}$,
    or a graph with treedepth at most~$O(a^3\log_2(a))$.  We also provide nearly-matching lower bounds.
\end{abstract} 

Planar graphs ``almost'' have bounded treewidth, in the following sense: For every assignment
of weights to vertices and for every positive integer~$a$, it is possible to delete vertices
of at most~$1/a$ fraction of the total weight so that the resulting subgraph has treewidth at most~$3a-3$.
Equivalently, there exists a probability distribution on subsets of vertices whose complement induces a subgraph of treewidth at most~$3a-3$,
such that each vertex belongs to a set sampled from this distribution with probability at most~$1/a$.
This property is the key ingredient of a number of approximation algorithms for planar graphs~\cite{baker1994approximation}.
To study this phenomenon more generally, Dvořák~\cite{twd} introduced the notion of fractional fragility of a graph class with
respect to a graph parameter.  Let us give the definitions we need to speak about this notion.

For~$\varepsilon>0$, we say that a probability distribution on the subsets of vertices of a graph~$G$ is \emph{$\varepsilon$-thin}
if for each vertex~$v$, the probability that~$v$ belongs to a set sampled from this distribution is at most~$\varepsilon$.
For example, we commonly use the following $(1/a)$-thin probability distribution. Suppose that sets $X_1, \dotsc, X_a\subseteq V(G)$ are
pairwise disjoint, and that~$t$ of these sets are empty.  We give to each non-empty set~$X_i$ the probability~$1/a$ and to the empty set the
probability~$t/a$.  All other sets are given probability~$0$.  We call this distribution the \emph{uniform distribution on~$\{X_1,\dotsc, X_a\}$}.

Let~$f$ be a graph parameter, that is, a function assigning to every
graph a non-negative real number such that isomorphic graphs are assigned the same
value.  We will generally consider parameters that are \emph{monotone} (satisfying that $f(H)\le f(G)$ whenever~$H$ is a
subgraph of~$G$), or at least \emph{hereditary} (satisfying $f(H)\le f(G)$ whenever~$H$ is an induced subgraph of~$G$).
For a real number~$b$ and a graph~$G$, let~$\mal{G}{f}{b}$ be the set of all subsets~$X\subseteq
V(G)$ such that $f(G[X])\le b$, and let~$\des{G}{f}{b}$ be the set of all
subsets~$Y\subseteq V(G)$ such that $f(G-Y)\le b$; thus~$Y\in \des{G}{f}{b}$ if
and only if $V(G)\setminus Y\in \mal{G}{f}{b}$.  For example, if~$\tw$ is the function
that to every graph assigns its treewidth, then~$\des{G}{\tw}{3a-3}$ is the set of vertex sets
whose complement induces a subgraph of treewidth at most~$3a-3$.

Let~$r\colon\mathbb{N}\to\mathbb{R}_0^+$ be a non-decreasing function.
A graph~$G$ is \emph{fractionally $f$-fragile at rate~$r$} if for every
positive integer~$a$, there exists a $(1/a)$-thin probability distribution
on~$\des{G}{f}{r(a)}$.  Of course, every graph is fractionally $f$-fragile at rate given by the constant function~$r(a)\colonequals f(G)$; so the notion is more
interesting for graph classes.  We say that a class of graphs is
\emph{fractionally $f$-fragile at rate~$r$} if each graph from the class is,
and we say that the class is \emph{fractionally $f$-fragile} if it is
fractionally $f$-fragile at some rate.  Coming back to the introductory example,
the class of planar graphs is known to be fractionally $\tw$-fragile at rate $r(a)\colonequals 3a-3$;
see Corollary~\ref{cor-tw-frag} below for details.

Graphs from a fractionally $f$-fragile classes can be viewed as being close to graphs for which the parameter~$f$ is bounded,
and this proximity can be useful when reasoning about their structural and quantitative properties.
There are also natural links to the theory classes of bounded expansion~\cite{twd}.
Furthermore, as we already mentioned in the introduction, the notion has algorithmic applications,
especially in the design of approximation algorithms.  A standard example is that of the independence number, which is hard to
approximate within a polynomial factor~\cite{apxindep} in general, but can be determined in linear
time over any class of graphs with bounded treewidth.  Consequently, there exists a polynomial-time
approximation scheme for the independence number over any class of graphs that is fractionally $\tw$-fragile
(assuming one can efficiently sample from the probability distribution guaranteed by the definition).

Fractional fragility is also related to generalizations of the (fractional)
chromatic number.  An \emph{$(f,b)$-coloring} of a graph~$G$ is an
assignment~$\varphi$ of colors to the vertices such that
$f(G[\varphi^{-1}(c)])\le b$ for every color~$c$, that is, such that each color class
belongs to $\mal{G}{f}{b}$.  We can now define~$\chi_{f,b}(G)$ as the
least number of colors in an $(f,b)$-coloring of~$G$.  For a class of
graphs~$\GG$, we naturally define~$\chi_f(\GG)$ as the smallest integer~$s$
such that for some positive integer~$b$, all graphs~$G\in\GG$ satisfy
$\chi_{f,b}(G)\le s$.  For example, let~$\star(G)$ be the maximum of the orders of the
components of the graph~$G$.  Then $\chi_{\star,1}(G)$ is just the ordinary chromatic number
of $G$, while in general, the parameter $\chi_{\star,b}(G)$ has been studied as the \emph{clustered chromatic number}~\cite{wood2018defective}.

Similarly to the way the fractional chromatic number is derived from the ordinary chromatic number~\cite{ScheinermanUllman2011},
we can also derive the fractional variant of this generalization.  A \emph{fractional $(f,b)$-coloring} of a
graph~$G$ is a function $\kappa\colon\mal{G}{f}{b}\to [0,1]$ such that for each vertex~$v\in V(G)$,
$$\sum_{Y\in\mal{G}{f}{b}, v\in Y} \kappa(Y)\ge 1;$$
the \emph{number of colors~$|\kappa|$} used by this
coloring is~$\sum_{Y\in\mal{G}{f}{b}} \kappa(Y)$. We define~$\chi'_{f,b}(G)$ to
be the infimum of~$|\kappa|$ over all fractional $(f,b)$-colorings~$\kappa$
of~$G$.  For a class $\GG$ of graphs, we define $\chi'_f(\GG)$ as the infimum
of the real numbers $s$ such that for some positive integer~$b$, all graphs~$G\in\GG$ satisfy
$\chi'_{f,b}(G)\le s$.

Note that unlike the ordinary fractional chromatic number case,
this can indeed be a proper infimum: as $b$ increases, the fractional $(f,b)$-coloring may need
fewer colors, converging to but never reaching $\chi'_f(\GG)$.  This motivates the following definition
that captures the rate of the convergence.  For a real number~$c$ and a
function~$r\colon\mathbb{N}\to\mathbb{R}_0^+$, we say that a class of
graphs~$\GG$ is \emph{fractionally $f$-colorable by $c$ colors at rate~$r$} if for every
integer~$a\ge 1$, every graph~$G\in \GG$ satisfies~$\chi'_{f,r(a)}(G)\le c+1/a$.
As we will see below (Lemma~\ref{lemma-mot}), fractional $f$-fragility is equivalent
to fractional $f$-colorability by $1$ color, at a matching rate.

The previous treatment of fractional fragility~\cite{twd} was mostly qualitative.
In this paper, we focus on the quantitative aspect: the rate of fractional fragility
for various parameters and graph classes. Note that the rate is important in the applications,
as it determines, \emph{e.g.}, the multiplicative constants in the complexity of the approximation
algorithms.

In Section~\ref{sec-maxcs}, we consider the parameter~$\star$, the maximum component
size.  By Lemma~\ref{lemma-bstar}, only classes of graphs with bounded maximum degree can be fractionally $\star$-fragile.
In Theorem~\ref{thm-exp}, we prove that the rate on any class of graphs containing at least
all subcubic trees is at least exponential.  Conversely, we show that graphs of bounded treewidth (Corollary~\ref{cor-tw})
and planar graphs (Theorem~\ref{thm-star-planar}) with fixed maximum degree nearly match this lower bound.

In Section~\ref{sec-treedepth}, we turn our attention to another graph parameter, \emph{treedepth}.
This parameter naturally generalizes the component size, but fractional $\td$-fragility does not require bounded maximum degree.
In this setting, we obtain polynomial bounds on the rate for graphs of bounded treewidth (Theorem~\ref{thm-ub-tw-td}),
outerplanar graphs (Theorem~\ref{thm-ub-op}), and planar graphs (Corollary~\ref{cor-planar}), as well as matching or nearly-matching
lower bounds (Theorems~\ref{thm-lb-td} and~\ref{thm-lb-op-td}).

\section{Preliminaries}

In this section, we show some basic properties of the fractional $f$-fragility,
and present several auxiliary results we need in the rest of the paper.

\subsection{Basic properties of fractional fragility}
The relationship between fractional $f$-colorability and fractional $f$-fragility
is given by the following lemma.

\begin{lemma}\label{lemma-mot}
    Let~$r\colon\mathbb{N}\to\mathbb{R}_0^+$ be a non-decreasing function, and
    let~$r'\colon\mathbb{N}\to\mathbb{R}_0^+$ be defined by setting
    $r'(a)\coloneqq r(a+1)$ for every~$a\in\mathbb{N}$. 
    Let~$f$ be a graph parameter whose value is at most~$r(1)$ on the empty graph.
    A class~$\GG$ of graphs is fractionally $f$-fragile at rate~$r$ if and only if it is fractionally
    $f$-colorable by~$1$ color at rate~$r'$.
\end{lemma}
\begin{proof}
    Suppose first that~$\GG$ is fractionally $f$-fragile at rate~$r$.
    Hence, for any positive integer~$a$ and any graph~$G\in \GG$, there exists a $\tfrac{1}{a+1}$-thin probability
    distribution on~$\des{G}{f}{r(a+1)}$. 
    Recall that a subset of~$V(G)$ belongs to~$\des{G}{f}{r(a)}$ if and only if its complement belongs to~$\mal{G}{f}{r(a)}$.
    For~$Y\in \mal{G}{f}{r'(a)}$, let~$\kappa(Y)\coloneqq \tfrac{a+1}{a}\Prb(V(G)\setminus Y)$.  For
    each~$v\in V(G)$, we have
    \begin{align*}
    \sum_{Y\in\mal{G}{f}{r'(a)},v\in Y} \kappa(Y)&=\sum_{X\in\des{G}{f}{r'(a)},v\not\in X} \kappa(V(G)\setminus X)
    =\frac{a+1}{a}\sum_{X\in\des{G}{f}{r'(a)},v\not\in X} \Prb(X)\\
    &=\frac{a+1}{a}\Prb[v\not\in X]\ge\frac{a+1}{a}\Bigl(1-\frac{1}{a+1}\Bigr)=1,
    \end{align*}
    and thus~$\kappa$ is a fractional $(f,r'(a))$-coloring of~$G$ using
    $|\kappa|=\tfrac{a+1}{a}=1+1/a$ colors.  Since this holds for every positive integer~$a$
    and for all graphs in~$\GG$, the class~$\GG$ is $f$-colorable by~$1$ color at rate~$r'$.

    Conversely, suppose that~$\GG$ is $f$-colorable by $1$ color at rate~$r'$.
    Consider a positive integer~$a$ and a graph $G\in \GG$.  Note that setting $\Prb(V(G))\coloneqq1$
    and~$\Prb(X)\coloneqq0$ for all~$X\subsetneq V(G)$ gives a $1$-thin probability
    distribution on~$\des{G}{f}{r(1)}$, since $r(1)\ge f(G-V(G))$.  Hence, we
    can assume that~$a\ge2$.  Then there exists a fractional
    $(f,r'(a-1))$-coloring $\kappa$ with $|\kappa|\le 1+\tfrac{1}{a-1}$, from
    which one can obtain a $(1/a)$-thin probability distribution
    on~$\des{G}{f}{r(a)}$ by
    setting~$\Prb(X)\coloneqq\tfrac{a-1}{a}\kappa(V(G)\setminus X)$.  This
    shows that~$\GG$ is fractionally $f$-fragile at rate~$r$.
\end{proof}
Let us note the following necessary condition for fractional $f$-fragility.
We say that a graph~$G$ is \emph{$f$-breakable at rate~$r$} if for every positive integer~$a$,
there exists a set~$X\in\des{G}{f}{r(a)}$ of size at most~$|V(G)|/a$. The next observation readily
follows from the definitions by using the linearity of expectation.
\begin{observation}\label{obs-break}
If a graph~$G$ is fractionally $f$-fragile at rate~$r$, then it is also $f$-breakable at rate~$r$.
\end{observation}
A seminal result on $\star$-breakability dates back to the work of Lipton and Tarjan~\cite{lt80};
they proved it in the special case of planar graphs, however, they proof directly generalizes to any class with sufficiently small
balanced separators.
A \emph{separation} in a graph $G$ is a pair~$(A,B)$ of subsets of vertices of~$G$ such that
$V(G)=A\cup B$ and no edge of~$G$ has one end in~$A\setminus B$ and the other end in~$B\setminus A$;
that is,~$A\setminus B$ and~$B\setminus A$ are unions of the vertex sets of the components of~$G-(A\cap B)$.
The \emph{order} of the separation is~$|A\cap B|$.
The separation is \emph{balanced} if $|A\setminus B|\le \tfrac{2}{3}|V(G)|$ and $|B\setminus A|\le \tfrac{2}{3}|V(G)|$.
Let~$s\colon\mathbb{N}\to\mathbb{R}_0^+$ be a non-decreasing function.  A graph~$G$ has \emph{balanced $s$-separators}
if every induced subgraph~$H$ of~$G$ has a balanced separator of order at most~$s(|V(H)|)$.
\begin{theorem}[Lipton and Tarjan~\cite{lt80}]\label{thm-break}
    Let~$\beta$ be a positive real number in~$(0,1]$.  For every
function~$s(n)=O(n^{1-\beta})$, there exists a function~$r(a)=O(a^{1/\beta})$
such that every graph with balanced $s$-separators is~$\star$-breakable at
rate~$r$.
\end{theorem}

We should also note the following property, already observed in an earlier work~\cite{twd}.
\begin{lemma}\label{lemma-bstar}
Suppose that~$f$ is a monotone graph parameter that is unbounded on stars.
    Then every fractionally $f$-fragile class of graphs has bounded maximum
    degree.
\end{lemma}
\begin{proof}
Suppose that a class~$\GG$ of graphs is fractionally $f$-fragile at rate~$r$.
    Since~$f$ is unbounded on stars, there exist an integer~$k$ such that
    $f(K_{1,k})>r(3)$.  We show that all graphs in~$\GG$ have maximum degree at
    most~$3k-3$.  Suppose, on the contrary, that a graph~$G\in \GG$ contains
    a vertex~$v$ of degree at least~$3k-2$.  Choose a
    set~$X\in\des{G}{f}{r(3)}$ at random from a $(1/3)$-thin probability distribution.
    Consider the
    random variable~$R\coloneqq \deg(v)\cdot[v\in X]+|N(v)\cap X|$, where~$[v\in X]$ is~$1$ if~$v\in X$ and~$0$ otherwise. The
    linearity of expectation ensures that $\Ex[R]\le \tfrac{2}{3}\deg(v)$, and
    hence there exists~$X\in \des{G}{f}{r(3)}$ such that $\deg(v)\cdot[v\in X]+|N(v)\cap X|\le \tfrac{2}{3}\deg(v)$.
    Consequently, $v\not\in X$ and
    $|N(v)\cap X|\le \tfrac{2}{3}\deg(v)$, and
    thus~$\deg_{G-X}(v)\ge \lceil\deg(v)/3\rceil\ge k$.  It follows that~$K_{1,k}$
    is a subgraph of~$G-X$. As~$f$ is monotone, we deduce that $f(G-X)\ge f(K_{1,k})>r(3)$,
    which contradicts that $X\in \des{G}{f}{r(3)}$.
\end{proof}

A linear programming dual formulation of fragility leads to the following observation.
For an assignment~$w\colon V(G)\to\mathbb{R}_0^+$ of weights to vertices and a set~$X\subseteq V(G)$,
let~$w(X)\coloneqq\sum_{v\in X} w(v)$.
\begin{lemma}\label{lemma-lbgen}
Let~$G$ be a graph that is fractionally $f$-fragile at rate~$r$.  Let~$a$ be a positive
    integer and~$w\colon V(G)\to\mathbb{R}_0^+$ an assignment of weights
    to the vertices of~$G$.  Then there exists $X\subseteq V(G)$ such that
    $w(X)\le w(V(G))/a$ and $f(G-X)\le r(a)$.
\end{lemma}
\begin{proof}
Choose a set $X\in\des{G}{f}{r(a)}$ at random from a $(1/a)$-thin probability distribution.
    By the
    linearity of expectation, $\Ex[w(X)]\le w(V(G))/a$, and thus there
    exists~$X\in \des{G}{f}{r(a)}$ such that~$w(X)\le w(V(G))/a$;
    \emph{i.e.}, there exists~$X\subseteq V(G)$ such that $w(X)\le w(V(G))/a$
    and~$f(G-X)\le r(a)$.
\end{proof}

\subsection{Chordal graphs}

Due to the following well-known observation, when considering graphs of bounded treewidth,
it is often convenient to work in the setting of \emph{chordal graphs}, that is, graphs not
containing any induced cycles other than triangles.
\begin{observation}\label{obs-twcho}
Every graph has a chordal supergraph with the same set of vertices and the same treewidth.
Moreover, if $G$ is chordal, then $\tw(G)=\omega(G)-1$.
\end{observation}
Each chordal graph~$G$ has an
\emph{elimination ordering}: an ordering of the vertices of~$G$ such that the
neighbors of each vertex that precede it in the ordering induce a clique.
By Observation~\ref{obs-twcho}, in an elimination ordering of $G$, each vertex is preceded by at
    most~$\tw(G)$ of its neighbors. Moreover, for every induced path~$P$ in~$G$,
    the last vertex of $V(P)$ according to the elimination ordering must be an end-vertex of~$P$.
    In particular, this implies the following property.
\begin{observation}\label{obs-prec}
Let~$G$ be a connected chordal graph and let~$v$ be the first vertex in an elimination ordering $L$ of~$G$.
For each vertex $u\in V(G)\setminus\{v\}$, the vertex preceding $u$ on any shortest path from $v$ to $u$
also precedes $u$ in $L$.
\end{observation}

\noindent
The next observation is also based on this fact.
\begin{lemma}\label{lemma-compose}
Let~$G$ be a connected chordal graph, let~$v$ be the first vertex in an elimination ordering of~$G$, let~$i$ be a non-negative integer,
and let~$H$ be a connected subgraph of~$G$ at distance greater than~$i$ from~$v$.  Let~$K$ be the set of vertices of~$G$ at distance
exactly~$i$ from~$v$ that have a neighbor in~$V(H)$.  Then~$K$ induces a clique in~$G$.
\end{lemma}
\begin{proof}
Let~$x$ and~$y$ be distinct vertices belonging to~$K$ (and thus~$x\neq v\neq y$, since both~$x$ and~$y$ are at the same distance from~$v$)
and suppose for a contradiction that $xy\not\in E(G)$. Since~$H$ is connected,
there exists a path between~$x$ and~$y$ in~$G$ with all internal vertices in~$H$;
let~$Q$ be a shortest such path.  It follows that~$Q$ is an induced path.
    Let~$z$ be the last vertex of~$Q$ in the elimination ordering of~$G$.
    Since $Q$ is an induced path and the neighbors of~$z$ in~$Q$ form a clique,
    we conclude that~$z$ is one of the ends of~$Q$, say $z=y$ by symmetry of the roles played by~$x$ and~$y$.
    Let~$u$ be the neighbor of~$y$ in~$Q$; since $xy\not\in E(G)$, we have
    $u\in V(H)$.  Since the distance from~$v$ to~$y$ is~$i$ and the distance
    to~$u$ is greater than~$i$, there exists a shortest path~$P$ from~$v$
    to~$u$ passing through~$y$.  But both~$v$ and~$u$ precede~$y$ in the
    elimination ordering, and thus the last vertex of~$P$ in the elimination
    ordering is neither of the ends of~$P$.  This is a contradiction, since~$P$
    is an induced path.
\end{proof}

\subsection{Planar graphs and treewidth}

As we have mentioned in the introduction, planar graphs are fractionally $\tw$-fragile.
This is a well-known consequence of the fact that the treewidth of planar graphs is at most linear in their radius,
which follows from ideas of Robertson and Seymour~\cite{rs3} and Baker~\cite{baker1994approximation}.
The version we use, together with a short proof, can be found in a work by Eppstein~\cite[Lemma~4]{bib-eppstein99}.
\begin{theorem}\label{thm-rad}
Every planar graph of radius at most~$d$ has treewidth at most~$3d$.
\end{theorem}
The fractional $\tw$-fragility now follows by a standard layering argument~\cite{baker1994approximation,eppstein00},
which we restate in our notation.
\begin{corollary}\label{cor-tw-frag}
    The class of planar graphs is fractionally $\tw$-fragile at rate $r(a)=3a-3$.
\end{corollary}
\begin{proof}
Let~$G$ be a planar graph, without loss of generality connected, and let~$a$ be
    a positive integer, at least~$2$ since the statement is trivial for~$a=1$.
    Let~$v$ be an arbitrary vertex of~$G$ and for every non-negative
    integer~$i$, let~$L_i$ be the set of vertices of~$G$ at distance
    exactly~$i$ from~$v$.  For~$i\in \{0,\dotsc, a-1\}$, set~$X_i\coloneqq
    L_i\cup L_{i+a}\cup L_{i+2a}\cup \dotsb$ and consider any component~$C$ of
    the graph~$G-X_i$.  There is some integer~$j$ such that~$C$ contains only
    vertices at distance between~$i+ja+1$ and~$i+ja+a-1$ from~$v$.  Let~$G'$ be
    the graph obtained from~$G$ by deleting all vertices at distance at least
    $i+ja+a$ from~$v$ and by contracting all vertices at distance at most
    $\max(i+ja,0)$ from~$v$ to a single vertex $x$.  Clearly, $G'$ is a minor
    of~$G$, and thus~$G'$ is planar.  Moreover, every vertex of~$G'$ is at
    distance at most $a-1$ from~$x$ and~$C\subseteq G'$.  Consequently,
    $\tw(C)\le \tw(G')\le 3a-3$ by Theorem~\ref{thm-rad}.  Since this is the
    case for every component of~$G-X_i$, we have $\tw(G-X_i)\le 3a-3$, and
    thus~$X_i\in \des{G}{\tw}{3a-3}$.  Since the uniform distribution
    on~$\{X_0,\dotsc, X_{a-1}\}$ is $(1/a)$-thin, planar graphs are
    fractionally $\tw$-fragile at rate~$r(a)=3a-3$.
\end{proof}

Pilipczuk and Siebertz~\cite{polycen} demonstrated another relationship between
planar graphs and graphs of bounded treewidth.
Given a partition~$\PP$ of vertices of a
graph~$G$, let~$G/\PP$ be the graph obtained from~$G$ by contracting each
part of~$\PP$ to a single vertex and suppressing the arising loops and parallel
edges.  A path~$P$ in a graph $G$ is \emph{geodesic} if for every~$x,y\in
V(P)$, the distance between~$x$ and~$y$ in~$G$ is the same as their distance
in~$P$.  Pilipczuk and Siebertz~\cite{polycen} proved that every planar
graph~$G$ admits a partition~$\PP$ of its vertices such that~$G/\PP$ has
treewidth at most~$8$ and each part of~$\PP$ induces a geodesic path in~$G$.
We need a variation on this result, which follows from a result proved by
Dujmovi\'c \emph{et al.}~\cite[Theorem~16]{DJM+} by using a breadth-first
search tree for their tree~$T_0$.  We say that a partition~$\PP$ of vertices of~$G$
is \emph{trigeodesic} if every part of~$\PP$ induces in~$G$ a connected
subgraph whose vertex set is covered by at most three geodesic paths
of~$G$. The aforementioned theorem yields the following (the last part uses the
fact that $G/\PP$ is a minor of $G$ and thus planar, and hence $\omega(G/\PP)\le 4$).

\begin{theorem}\label{thm-trigeod}
For every plane triangulation~$G$, there exists a trigeodesic partition~$\PP$
of vertices of~$G$ such that $G/\PP$ is chordal. In particular,~$G/\PP$ has
treewidth at most~$3$.
\end{theorem}

\noindent


\section{Maximum component size}\label{sec-maxcs} 

Recall that~$\star(G)$ is the maximum of the orders of the components of the graph~$G$.
The parameter~$\chi_{\star,b}$ has been intensively studied under the name
\emph{clustered chromatic number}~\cite{wood2018defective}, and is among the most natural relaxations of
the chromatic number.  Clustered coloring specializes to the usual notion of
vertex coloring, in the sense that $\chi_{\star,1}(G)=\chi(G)$.

In the special case of planar graphs, clustered chromatic number is in general no better than ordinary chromatic
number: for every integer~$b$, there exists a planar graph~$G_b$ such that $\chi_{\star,b}(G_b)=4$.
These graphs~$G_b$ necessarily have unbounded maximum degree: Esperet and Joret~\cite{espjor} proved that
for every~$\Delta$, there exists~$b$ such that every planar graph~$G$ of maximum degree at most~$\Delta$ satisfies~$\chi_{\star,b}(G)\le 3$.
Moreover, the Hex lemma implies that this bound cannot be improved.
The situation is different in the fractional setting due to Lemma~\ref{lemma-mot}, since planar graphs of bounded maximum degree
are fractionally $\star$-fragile (the assumption of bounded maximum degree is necessary by Lemma~\ref{lemma-bstar}).
In fact, Dvořák~\cite{twd} proved fractional $\star$-fragility in much greater generality, for all classes of bounded maximum degree
with strongly sublinear separators.

\begin{theorem}[Dvořák~\cite{twd}]\label{thm-star-fr}
    Let~$\beta$ be a real number in~$(0,1]$.  For every function
    $s(n)=O(n^{1-\beta})$ and every integer $\Delta$, there exists a function~$r$ such that every graph
    with balanced $s$-separators and maximum degree at most $\Delta$ is fractionally $\star$-fragile at rate~$r$.
\end{theorem}

Let us remark that the argument used to prove Theorem~\ref{thm-star-fr} gives a very bad bound on the rate~$r$, especially compared
to the polynomial $\star$-breakability bound from Theorem~\ref{thm-break}.  As shown by Lipton and Tarjan~\cite{lt79}, planar graphs have balanced $s$-separators for~$s(n)=3\sqrt{n}$,
and thus they are $\star$-breakable at rate~$O(a^2)$.  Considering Observation~\ref{obs-break}, it is natural to ask whether
(subject to a bound on the maximum degree) planar graphs are also fractionally $\star$-fragile at quadratic rate~$O(a^2)$.
As our first result, we show that this is not the case, even for much more restricted graph classes.
\begin{theorem}\label{thm-exp}
Let~$\Delta\ge 3$ be an integer and let~$\GG$ be a class of graphs that contains all trees of maximum degree at most~$\Delta$.
If~$\GG$ is fractionally $\star$-fragile at rate~$r$, then~$r(a)\ge (\Delta-1)^{a-3}$ for every integer~$a\ge 4$.
\end{theorem}
\begin{proof}
    Fix an integer~$a\ge4$.  Let~$T$ be the complete rooted $(\Delta-1)$-ary
    tree of depth~$d$ (the root has depth~$0$ and the leaves have depth~$d$,
    every non-leaf vertex has exactly~$\Delta-1$ children), where~$d\ge 3a-1$.
    We aim to use Lemma~\ref{lemma-lbgen}.
    For every vertex~$v\in V(T)$ at depth~$k$, let~$w(v)\coloneqq(\Delta-1)^{-k}$, so
    $w(V(T))=d+1$.  We prove that, for every set~$X\subseteq V(T)$
    with~$w(X)\le (d+1)/a$, the forest~$T-X$ contains a component with at
    least~$(\Delta-1)^{a/(1+a/(d+1))-3}$ vertices.
    
    Consider any set~$X\subseteq
    V(T)$ such that $w(X)\le (d+1)/a$.  Let~$X'$ consist of~$X$ and the root
    of~$T$; we have $w(X')\le 1+(d+1)/a$.  For a vertex~$v\in X'$, let~$C_v$
    be the set of all descendants of~$v$ in~$T$ (including~$v$ itself) that
    can be reached without passing through another vertex of~$X'$.  Then
    $\{C_v\,:\,v\in X'\}$ is a partition of~$V(T)$.  For~$v\in X'$,
    set~$r(v)\coloneqq w(C_v)/w(v)$.  We have
    \begin{align*}
        \frac{\sum_{v\in X'}w(v)r(v)}{w(X')}&=\frac{\sum_{v\in X'} w(C_v)}{w(X')}=\frac{w(V(T))}{w(X')}\\
                    &\ge\frac{d+1}{(d+1)/a+1}=\frac{1}{1+a/(d+1)}\cdot a.
    \end{align*}
    Let~$a'\coloneqq\tfrac{1}{1+a/(d+1)}\cdot a$, and note that $a'\ge \tfrac{3}{4}a\ge 3$
    because~$d\ge3a-1$ and~$a\ge4$.  Since the left side of the above inequality is a
    weighted average of the values~$r(v)$ for~$v\in X'$, there exists~$v\in X'$ such that $r(v)\ge a'$,
    and thus 
    $w(C_v)\ge a'w(v)$.
    
    For each non-negative integer~$i$, let~$n_i$ be the number
    of vertices in~$C_v$ whose depth is by~$i$ larger than the depth of~$v$, so that
    $w(C_v)=w(v)\sum_{i\ge 0} (\Delta-1)^{-i}\cdot n_i$, and thus~$a'\le \sum_{i\ge
    0} (\Delta-1)^{-i}\cdot n_i$.  Subject to this inequality and to the constraints~$n_i\le
    (\Delta-1)^i$ for every~$i$, the value $|C_v|=\sum_{i\ge 0} n_i$ is
    minimized when $n_i=(\Delta-1)^i$ for~$i\in\{0,\dotsc,m-1\}$ and~$n_i=0$ for~$i\ge
    m+1$ where~$m=\lfloor a'\rfloor\ge 3$ (as can be seen by a standard
    weight-shifting argument).  It follows that 
    \[
        |C_v|\ge \sum_{i=0}^{m-1}
        (\Delta-1)^i=\frac{(\Delta-1)^m-1}{\Delta-2}\ge\frac{(\Delta-1)^{a'-1}-1}{\Delta-2}\ge
        (\Delta-1)^{a'-2}+1.
    \] 
    Consequently,~$T[C_v]-v$ has a component with at least $(\Delta-1)^{a'-3}$
    vertices (since~$v$ has~$\Delta-1$ children in~$T$), giving the same lower
    bound on~$\star(T-X)$.  By Lemma~\ref{lemma-lbgen}, we conclude
    that~$r(a)\ge (\Delta-1)^{a'-3}$.  Because this inequality holds for
    all~$d\ge 3a-1$ and~$\lim_{d\to\infty} a'=a$, the
    statement of the lemma follows.
\end{proof}

Conversely, many interesting graph classes, including planar graphs, nearly
match the lower bound provided by Theorem~\ref{thm-exp}.  We start by an argument
for graphs with bounded treewidth.  We use the following well-known fact~\cite[(2.6)]{rs2}.
\begin{observation}\label{obs-tw-split}
    Let~$k$ be an integer.  If~$G$ is a graph of treewidth less than~$k$ and~$Z$ a subset of
    vertices of~$G$,  then~$G$ has a separation~$(D,B)$ of order at most~$k$
    such that $|Z\setminus D|\le \tfrac{2}{3}|Z|$ and $|Z\setminus B|\le
    \tfrac{2}{3}|Z|$.
\end{observation}
\noindent
Iterating this splitting procedure, we obtain the following generalization.
\begin{lemma}\label{lemma-tw-split}
    Let~$k,s$ and~$p$ be positive integers such that~$s\ge 12k$.
    If~$G$ is a graph of treewidth less than~$k$ and~$W$ a subset of vertices of~$G$ of order at most~$ps$,
    then there exists a set~$C\subseteq V(G)$ and non-empty sets~$A_1, \dotsc, A_t\subseteq V(G)$ for some~$t<6p$ such that
\begin{itemize}
\item[(i)] $|C|< 6pk$;
\item[(ii)] $|A_i\cap (C\cup W)|\le s$ for each~$i\in\{1,\dotsc,t\}$;
\item[(iii)] $G=G[A_1]\cup \cdots\cup G[A_t]$; and
\item[(iv)] $A_i\cap A_j\subseteq C$ if~$1\le i<j\le t$.
\end{itemize}
\end{lemma}
\begin{proof}
    We inductively define~$\AA_i$ and~$C_i$ for~$i\in\{0,\dotsc,t-1\}$.
    Let~$\AA_0\coloneqq\{V(G)\}$ and~$C_0\coloneqq \varnothing$.
    For~$i\ge 0$, if there exists~$X_i\in\AA_i$ such that $|X_i\cap (C_i\cup W)|>s$, we apply Observation~\ref{obs-tw-split}
    to~$G[X_i]$ with the subset~$Z_i\coloneqq X_i\cap (C_i\cup W)$ of vertices, obtaining a
    separation~$(D_i,B_i)$ of~$G[X_i]$ of order at most~$k$; and we let~$\AA_{i+1}\coloneqq (\AA_i\setminus\{X_i\})\cup\{D_i,B_i\}$
    and $C_{i+1}\coloneqq C_i\cup (D_i\cap B_i)$.  If no such element~$X$ exists, the procedure stops and we set~$t\coloneqq i+1$,
    $C\coloneqq C_i$ and~$\{A_1,\dotsc, A_t\}\coloneqq \AA_i$.  Assuming the construction stops, it is clear the conditions~(ii),~(iii) and~(iv) hold.
    Since $|C_{i+1}\setminus C_i|\le k$ for~$i\in\{0,\dotsc,t-2\}$, it suffices to argue that the construction stops with $t<6p$.
    Without loss of generality, we can assume that the construction does not stop in the first step, \emph{i.e.}, that~$|W|>s$.

    If~$0\le i\le t-1$ and~$X\subseteq V(G)$, we let $\partial_i X\coloneqq |X\cap (C_i\cup W)|$.
    Suppose that~$i\le t-2$.  Note that if $X\in\AA_i$ and~$X\neq X_i$, then $X\cap C_{i+1}=X\cap C_i$,
    since $C_{i+1}\setminus C_i\subseteq X_i\setminus C_i$ is disjoint from~$X$; hence, $\partial_{i+1} X=\partial_i X$.
    By the choice of the separation $(D_i,B_i)$, we have
    \[\partial_{i+1}D_i\ge |D_i\cap Z_i|=|Z_i|-|Z_i\setminus D_i|\ge |Z_i|/3>s/3,\]
    and symmetrically $\partial_{i+1}B_i>s/3$.  We conclude that if~$0\le i\le t-1$, then
    \[\sum_{X\in \AA_i} \partial_i X>|\AA_i|s/3=(i+1)s/3.\]
    On the other hand,
    \[\partial_{i+1} D_i+\partial_{i+1} B_i\le \partial_i X_i + 2|D_i\cap B_i|=\partial_i X_i + 2k,\]
    and thus
    $$\sum_{X\in\AA_{i+1}} \partial_{i+1} X\le 2k+\sum_{X\in\AA_i} \partial_i X.$$
    By induction, we conclude that for~$i\in\{0,\dotsc,t-1\}$, we have
    \[\sum_{X\in \AA_i} \partial_i X\le |W|+2ik\le ps+2ik\le (p+i/6)s.\]
    Combining the inequalities, we obtain~$(i+1)/3<(p+i/6)$, and hence~$i<6p-2$.
    Consequently, the construction stops with~$t<6p$.
\end{proof}

\begin{corollary}\label{cor-tw-split}
    Let~$k,s$ and~$p$ be positive integers such that~$s\ge 12k$.
    If~$G$ is a graph of treewidth less than~$k$ and~$W$ a subset of vertices of~$G$ of order at most~$ps$,
    then there exists a set~$C\subseteq V(G)\setminus W$ and a partition $E_1, \dotsc, E_t$ of~$V(G)\setminus (C\cup W)$ for some~$t<6p$ such that
\begin{itemize}
\item[(i)] $|C|< 6pk$;
\item[(ii)] for each~$i\in\{1,\dotsc,t\}$, at most~$s$ vertices in~$C\cup W$ have a neighbor in~$E_i$;
\item[(iii)] for each non-isolated vertex~$v$ in~$C$, either~$v$ has a neighbor in~$C\cup W$, or in at least two of the sets $E_1, \dotsc, E_t$;
\item[(iv)] $G=G[C\cup W\cup E_1]\cup\dotsb\cup G[C\cup W\cup E_t]$.
\end{itemize}
\end{corollary}
\begin{proof}
Apply Lemma~\ref{lemma-tw-split} and replace~$C$ by~$C\setminus W$ if necessary, so that~$C\cap W=\varnothing$.
    Let~$E_i\coloneqq A_i\setminus (C\cup W)$ for $i\in\{1,\dotsc,t\}$, and remove from the list $E_1, \dotsc, E_t$ the empty sets.
Finally, if a vertex~$v\in C$ has neighbors in~$E_i$ and only in~$E_i$ for some~$i\in\{1,\dotsc, t\}$, then we can replace~$C$ by~$C\setminus \{v\}$
and~$E_i$ by~$E_i\cup \{v\}$.
\end{proof}

A \emph{tree partition}~$(T,\beta)$ of a graph~$G$ consists of a tree~$T$ and a
function~$\beta$ that to each vertex of~$T$ assigns a subset of vertices
of~$G$, such that
\begin{itemize}
\item the sets~$\beta(v)$ for~$v\in V(T)$ are pairwise disjoint and form a
    partition of~$V(G)$; and 
\item if two vertices~$x$ and~$y$ of~$T$ are not adjacent, then~$G$ does not contain any edge with
    one end in~$\beta(x)$ and the other in~$\beta(y)$.
\end{itemize}
Equivalently, the graph obtained from~$G$ by contracting each set~$\beta(x)$ for~$x\in V(T)$ to a single vertex
(and removing loops and multiple edges) is a subgraph of~$T$.  In a
\emph{rooted tree partition}, the tree~$T$ is additionally rooted.  For a
subtree~$S\subseteq T$, let~$\beta(S)\coloneqq\bigcup_{v\in V(S)} \beta(v)$.
For every integer~$a$, the \emph{depth-$a$ order} of the rooted tree partition
is the maximum of~$|\beta(S)|$ over all subtrees~$S$ of~$T$ of depth at most~$a-2$.
We use the following simple observation.
\begin{lemma}\label{lemma-tp-st}
Let~$r\colon\mathbb{N}\to\mathbb{R}_0^+$ be a non-decreasing function.
If for every positive integer~$a$, the graph~$G$ admits a rooted tree
    partition~$(T_a,\beta)$ of depth-$a$ order at most~$r(a)$, then~$G$ is
    fractionally $\star$-fragile at rate~$r$.
\end{lemma}
\begin{proof}
    Let~$a$ be a positive integer.  For each~$i\in \{0,\dotsc, a-1\}$,
    let~$L_i$ be the set of vertices of~$T_a$ with a depth belonging
    to~$\{i+ja\,:\,j\in\mathbb{N}_0\}$ and set~$X_i\coloneqq\bigcup_{v\in L_i}
    \beta(v)$.  By the definition of a tree partition, the vertex set of each
    component of~$G-X_i$ is contained in~$\beta(S)$ for a component~$S$
    of~$T_a-L_i$.  Each component of~$T_a-L_i$ is a tree of depth at most~$a-2$,
    and hence~$|\beta(S)|\le r(a)$.  Consequently,~$\star(G-X_i)\le
    r(a)$, and thus~$X_i\in \des{G}{\star}{r(a)}$.  Note that~$X_i\cap X_j=\varnothing$ if~$i\neq j$.
    Considering the uniform distribution on~$\{X_0,\dotsc, X_{a-1}\}$ (which is $(1/a)$-thin),
    we conclude that~$G$ is fractionally $\star$-fragile at rate~$r$.
\end{proof}
For example, if~$\Delta\ge3$ and~$T$ is a tree of maximum degree at
most~$\Delta$, then we can root~$T$ and define $\beta(v)\coloneqq\{v\}$ for
every~$v\in V(T)$, thereby obtaining a rooted tree partition of~$T$ of depth-$a$
order~$O((\Delta-1)^{a-2})$. It thus follows from Lemma~\ref{lemma-tp-st} that
trees of maximum degree at most~$\Delta$ are fractionally $\star$-fragile at
rate~$O((\Delta-1)^{a-2})$, essentially matching the bound from
Theorem~\ref{thm-exp}.

We now construct good tree partitions for graphs of bounded treewidth and maximum degree.
\begin{lemma}\label{lemma-goodtp}
Let~$a,b,k$ and~$\Delta$ be positive integers with~$\Delta\ge3$ and~$a\ge b$.
If~$G$ is a connected graph with treewidth less than~$k$ and maximum degree
at most~$\Delta$, then~$G$ admits a rooted tree partition~$(T,\beta)$ of
    depth-$a$ order at most~$12k(\Delta-1)^a\bigl((\Delta-1)^{b-1}+6^{a/b}\bigr)$.
\end{lemma}
\begin{proof}
Let~$s\coloneqq12k$.  We construct the tree partition starting from the root and adding
    children as described below. To every vertex~$v$ of~$T$ will be
    associated, in addition to~$\beta(v)$, three sets,
    namely~$\sigma(v)$,~$\gamma(v)$ and~$\kappa(v)$.  When considering a
    vertex~$v$ with parent~$z$ in~$T$, two of these will already have been
    defined in one of the previous steps: the set~$\sigma(v)\subseteq V(G)\setminus\beta(z)$,
    which at the end of the construction will be equal to~$\beta(S)$ for the
    subtree~$S$ of~$T$ consisting of~$v$ and all its descendants, and the
    set~$\gamma(v)\subseteq\beta(z)$, which is of size at
    most~$(\Delta-1)^{b-1}s$ and such that in~$G$, all neighbors of vertices
    from~$\sigma(v)$ are contained in~$\sigma(v)\cup\gamma(v)$, and each vertex
    in~$\gamma(v)$ has at most $\Delta-1$ neighbors in~$\sigma(v)$.
    The set~$\kappa(v)$, which must be contained in~$\beta(v)$, is defined when~$v$
    is considered; its role becomes clear later.

Clearly, we can assume that~$G$ has at least three vertices.  For the root~$r$
    of~$T$, we start the construction by letting~$\beta(r)$ consist of two
    adjacent vertices of~$G$ and~$\kappa(r)\coloneqq\varnothing$, adding a
    child $u$ of~$r$ to~$T$, and setting~$\sigma(u)\coloneqq V(G)\setminus
    \beta(r)$ and~$\gamma(u)\coloneqq \beta(r)$.  Suppose now that the
    construction reaches a vertex~$v$ of~$T$ with parent~$z$.  Let~$W$ be the
    set of vertices in~$\sigma(v)$ that have a neighbor (in~$G$)
    in~$\gamma(v)$.  If $|W|\le (\Delta-1)^{b-1}s$, we let~$\beta(v)\coloneqq W$
    and~$\kappa(v)\coloneqq \varnothing$; when $\sigma(v)=W$, then~$v$ is a leaf of~$T$,
    otherwise, we add a child~$x$ to~$v$ and set $\sigma(x)\coloneqq \sigma(v)\setminus
    W$ and~$\gamma(x)\coloneqq W$. Notice that if~$y\in\sigma(x)$, then all neighbors
    of~$y$ in~$G$ are contained in~$\sigma(v)\cup\gamma(v)$, since~$\sigma(x)\subseteq\sigma(v)$.
    Moreover, because~$y\notin W$ we know that~$y$ has no neighbor in~$\gamma(v)$, and hence
    all neighbors of~$y$ are contained in~$\sigma(v)=\sigma(x)\cup\gamma(x)$.
    Let us also point out that a vertex~$w\in\gamma(x)$
    has less than~$\Delta$ neighbors in~$\sigma(x)$, because~$w$ has a neighbor in~$\gamma(v)$,
    which is disjoint from~$\sigma(v)$.

    Let us consider the case that
    $|W|>(\Delta-1)^{b-1}s$; in this case, we say that~$v$ is a \emph{branching
    vertex}.  Since $|\gamma(v)|\le (\Delta-1)^{b-1}s$ and each vertex
    in~$\gamma(v)$ has at most~$\Delta-1$ neighbors in~$\sigma(v)$, we have
    $|W|\le (\Delta-1)^bs$.  Let~$C, E_1, \dotsc, E_t\subseteq \sigma(v)$ be
    the sets obtained by applying Corollary~\ref{cor-tw-split} to~$G[\sigma(v)]$
    and~$W$, with~$p$ being~$(\Delta-1)^{b}$.
    We set~$\beta(v)\coloneqq W\cup C$, $\kappa(v)\coloneqq C$, we add~$t$
    children~$u_1, \dotsc, u_t$ to~$v$, and set $\sigma(u_i)\coloneqq
    E_i$ and~let $\gamma(u_i)$ consist of all vertices in $W\cup C$ with a neighbor in~$E_i$
    for~$i\in\{1,\dotsc,t\}$.
    Let us point out that~$|\gamma(u_i)|\le s\le(\Delta-1)^{b-1}s$ by property~(ii) from Corollary~\ref{cor-tw-split},
    and that vertices of $\sigma(u_i)$ only have neighbors in $\sigma(u_i)\cup \gamma(u_i)$ by property~(iv) from Corollary~\ref{cor-tw-split}.
    Let us also remark that each vertex in~$\gamma(u_i)$
    has at most~$\Delta-1$ neighbors in~$\sigma(u_i)$, since~$G$ has maximum degree at most~$\Delta$,
    each vertex in~$W$ has a neighbor in~$\gamma(v)$, and due to the property~(iii) from Corollary~\ref{cor-tw-split}
    for vertices in~$C\cap \gamma(u_i)$.

Note that since $|\gamma(v)|\le s$ when~$v$ is the child of a branching vertex,
    and~$|\beta(v)|\le (\Delta-1)|\gamma(v)|$ when~$v$ is not a branching
    vertex, if~$x$ and~$y$ are two distinct branching vertices and~$x$ is an ancestor of~$y$,
    then the depth of~$x$ is by at least~$b$ larger than the depth of~$y$.
    Note also that every branching vertex~$x$ has less than~$6(\Delta-1)^b$
    children and satisfies~$|\kappa(x)|< 6(\Delta-1)^bk$.

The described construction clearly results in a rooted tree partition of~$G$.
    Let us now consider any subtree~$S$ of~$T$ of depth at most~$a-2$, with
    root~$w$.  The \emph{level} of a branching vertex~$x$ of~$S$ is the number
    of branching vertices on the path from~$x$ to~$w$, excluding~$x$ itself;
    hence, each branching vertex has level at most~$\lfloor (a-2)/b\rfloor\le
    \lfloor a/b\rfloor$.  The number of branching vertices of~$S$ of level~$i$
    is at most $\bigl(6(\Delta-1)^b\bigr)^i$.  If~$w$ is the root of~$T$, then
    let~$B\coloneqq \beta(w)$, otherwise let~$B$ be the set of vertices in~$\sigma(w)$
    with a neighbor in~$\gamma(w)$; in either case, we have~$|B|\le (\Delta-1)^{b}s$.  Note
    that each vertex in~$\beta(S)$ is either at distance at most~$a-2$
    from~$B$, or at distance at most~$a-2-b\cdot i$ from a vertex
    in~$\kappa(x)$ for some branching vertex~$x\in V(S)$ of level~$i$.
    Therefore, we have
\begin{align*}
|\beta(S)|&\le (\Delta-1)^{a-1}\left((\Delta-1)^{b}s+\sum_{i=0}^{\lfloor a/b\rfloor}\bigl(6(\Delta-1)^b\bigr)^i(\Delta-1)^{-bi}k\right)\\
&=k(\Delta-1)^{a-1}\left(12(\Delta-1)^{b-1}+\sum_{i=0}^{\lfloor a/b\rfloor}6^i\right)\\
    &\le 12k(\Delta-1)^a\bigl((\Delta-1)^{b-1}+6^{a/b}\bigr),
\end{align*}
as required.
\end{proof}

\noindent
We now combine Lemmas~\ref{lemma-tp-st} and~\ref{lemma-goodtp}, choosing~$b=\Theta(\sqrt{a})$ in the latter.
\begin{corollary}\label{cor-tw}
    Let~$k$ and~$\Delta$ be positive integers with~$\Delta\ge3$.
The class of graphs of treewidth less than~$k$ and maximum degree at most~$\Delta$
    is fractionally $\star$-fragile at rate~$r(a)=k(\Delta-1)^{a+O(\sqrt{a})}$.
\end{corollary}

\noindent
The result can be extended to planar graphs using their fractional $\tw$-fragility; that is, by
combining Corollaries~\ref{cor-tw-frag} and~\ref{cor-tw}.
\begin{theorem}\label{thm-star-planar}
    For every integer~$\Delta\ge3$, the class of
    planar graphs with maximum degree at most~$\Delta$ is fractionally $\star$-fragile at rate~$r(a)=(\Delta-1)^{a+O(\sqrt{a})}$.
\end{theorem}
\begin{proof}
Let~$G$ be a planar graph of maximum degree at most~$\Delta$.
Consider an integer~$a>256$, let~$a'\coloneqq\lceil2^{\sqrt{a}}\rceil$, let~$a''\coloneqq a+1$ and note that~$1/a'+1/a''<1/a$.
Choose~$X\in \des{G}{\tw}{3a'-3}$ at random from the $(1/a')$-thin probability distribution given by Corollary~\ref{cor-tw-frag}.
Then~$G-X$ is a planar graph of treewidth less than~$3a'-2$ and maximum degree at most~$\Delta$.
    Choose~$Y\in \des{(G-X)}{\star}{(3a'-2)(\Delta-1)^{a''+O(\sqrt{a''})}}$ from the $(1/a'')$-thin probability distribution given by Corollary~\ref{cor-tw},
and let~$Z\coloneqq X\cup Y$.  Then $\star(G-Z)=\star((G-X)-Y)\le(3a'-2)(\Delta-1)^{a''+O(\sqrt{a''})}=(\Delta-1)^{a+O(\sqrt{a})}$.
Consequently, choosing~$Z$ in this way gives a probability distribution on~$\des{G}{\star}{(\Delta-1)^{a+O(\sqrt{a})}}$,
and $\Prb[v\in Z]\le \Prb[v\in X]+\Prb[v\in Y]\le 1/a'+1/a''<1/a$ for every~$v\in V(G)$.
We conclude that~$G$ is fractionally $\star$-fragile at rate~$(\Delta-1)^{a+O(\sqrt{a})}$.
\end{proof}

\section{Treedepth}\label{sec-treedepth} 

By Lemma~\ref{lemma-bstar}, we cannot hope to extend the results on fractional
$\star$-fragility to any class with unbounded maximum degree.  The natural
parameter to consider in graphs with unbounded maximum degree is the
\emph{treedepth}~\cite{nesbook}: firstly, stars have treedepth at most~$2$, and secondly, a
connected graph of maximum degree at most~$\Delta$ and treedepth at most~$d$
has at most~$\Delta^d$ vertices, thus giving us about as good a relationship
to~$\star$ as one may hope for in the case where the maximum degree is bounded
from above.  The \emph{treedepth}~$\td(G)$ of a graph~$G$ is the minimum integer $d$
for which there exists a rooted tree~$T$ of depth at most~$d-1$ with vertex set~$V(G)$ such that every
edge of~$G$ joins a vertex to one of its ancestors or descendants in~$T$.

Given Corollary~\ref{cor-tw} and the relationship between~$\star$ and~$\td$
outlined above, one could perhaps hope that graphs of bounded treewidth are
fractionally $\td$-fragile at a linear rate.  However, this is not the case.
For the simplicity of presentation, we only give the counterargument for the case
of graphs of treewidth two, but it can be naturally generalized to show that
the class of all graphs of treewidth at most~$t$ cannot be fractionally
$\td$-fragile at rate better than~$\Omega(a^t)$.

\begin{figure} 
    \begin{center}
\begin{tikzpicture}[vertex/.style={circle, draw=black, fill=black, inner sep=0.5pt, minimum
        size=6pt},
        1-vertex/.style={circle, draw=black, fill=black, inner sep=0mm, minimum
        size=2.5pt},%
       2-vertex/.style={circle, draw=black, fill=black, inner sep=0mm, minimum
        size=4pt},%
       4-vertex/.style={regular polygon,regular polygon sides=4, draw=black, fill=black, inner sep=0mm, minimum
        size=8pt},%
       edge/.style={thick},%
       scale=.8
      ]
\def\a{3.8}
\def\b{1.5}
\foreach \t\n in {0/0,1/.85,2/2.15,3/3} {
    \draw (\a*\n*\unit,0) node[vertex] (a\t) {};
\begin{scope}[rotate=-(4-2138/731*\n-1200/731*\n*\n+800/731*\n*\n*\n)*90]
\foreach \j in {0,1,2,3} {
  \path (a\t)--++(-1,-\b*\unit+\j*2/3*\b*\unit) node[2-vertex] (c\t\j) {};
   \draw[edge] (c\t\j)--(a\t);
   \foreach \k in {0,1,2,3} {
    \path (c\t\j)--++(-1,-.3*\unit+\k*.2*\unit) node[1-vertex] (d\t\j\k) {};
    \draw[edge] (d\t\j\k)--(c\t\j);
} \draw[edge] (d\t\j0)--(d\t\j3); } \draw[edge] (c\t0)--(c\t3);
\end{scope}
}
\draw[edge] (a0)--(a3);
\draw (1.5*\a*\unit,-\a/2*\unit) node[4-vertex] (h) {};
\foreach \i in {0,1,2,3} { \draw[edge] (a\i)--(h); }
    \path (c11)--(c12) node[midway,above=4mm] (c){};
    \draw[dashed] (c) circle (\a*.6*\unit);
    \draw (a1) node[below left] {$x$};
\end{tikzpicture}
    \end{center}
        \caption{The graph~$T_3(P_4)$: the handle is the vertex represented by a square,
        and the circled part is the jug of the vertex~$x$.}\label{fig-td}
\end{figure}
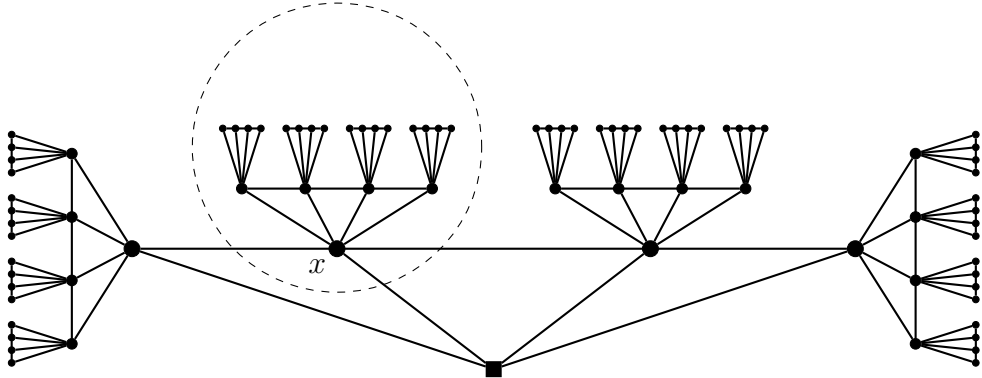

For a graph~$H$ and a
non-negative integer~$d$, let~$T_d(H)$ be the graph inductively defined as follows:
$T_0(H)$ is the graph consisting of a single vertex~$v$, which we call the
\emph{handle} of~$T_0(H)$.  For~$d\ge 1$, let~$T_d(H)$ be the graph obtained
from~$H$ by adding, for each~$x\in V(H)$, a copy of~$T_{d-1}(H)$ and
identifying its handle with~$x$, and finally adding a vertex $v$ adjacent to
all vertices of~$H$; the vertex~$v$ is the \emph{handle} of~$T_d(H)$.
Figure~\ref{fig-td} gives a representation of~$T_3(H)$ when~$H$ is the
$4$-vertex path~$P_4$.
Note that for each vertex~$x$ of~$T_d(H)$, there is a unique
index~$i\in\{0,\dotsc,d\}$ such that~$x$ is the handle of a copy of~$T_i(H)$;
let us call this copy the \emph{jug} of~$x$.  Given a non-identically-zero
function~$w\colon V(H)\to\mathbb{R}_0^+$ assigning weights to vertices of~$H$,
let~$w_d\colon V(T_d(H))\to\mathbb{R}_0^+$ be defined as follows.  For the
handle~$v$, we set~$w_d(v)\coloneqq1$, and when~$d\ge 1$, for each vertex~$x$
in the copy of~$H$ we set~$w_d(x)\coloneqq w(x)/w(V(H))$, and in the copy
of~$T_{d-1}(H)$ attached at~$x$, we set the weights according to~$w_d(x)\cdot
w_{d-1}$.  It may help to follow the sequel to notice that~$w_d(V(T_d(H)))=d+1$ for
every non-negative integer~$d$ and every graph~$H$.  Further, if the jug~$J$
of~$x$ in~$T_d(H)$ is a copy of~$T_i(H)$ with~$i\ge1$, then $\sum_{v\in
N_J}w_d(v)=w_d(x)$ where~$N_J$ is the set of neighbors of~$x$ in~$T_d(H)$ that
belong to~$J$.

Let~$B_d$ be the complete binary tree of depth~$d$
(let us remark that~$B_d=T_d(2K_1)$, where~$2K_1$ is the graph with two
vertices and no edge), and let~$t_d\colon V(B_d)\to\mathbb{R}_0^+$ be the
weight function assigning to each vertex of depth~$i$ the weight~$2^{-i}$
(so~$t_d=w_d$ for the weight function~$w$ assigning to both vertices of~$2K_1$
the same weight).  Let us start with an observation on complete binary subtrees
in heavy subsets of~$B_d$.  For a graph~$H$ with a handle~$h$, we say that a
graph~$G$ with a vertex~$s$ \emph{contains a minor of~$H$ rooted in~$s$} if
there exists an assignment~$\mu$ of pairwise disjoint non-empty sets of
vertices of~$G$ to the vertices of~$H$, such that
\begin{itemize}
    \item $s\in \mu(h)$;
    \item for each vertex~$v\in V(H)$, the graph~$G[\mu(v)]$ is connected;
    \item for each edge~$uv\in E(H)$, there exists an edge of~$G$ with one end
        in~$\mu(u)$ and the other end in~$\mu(v)$.
\end{itemize}
The sets~$\mu(v)$ are called the \emph{bags} of the minor.

\begin{lemma}\label{lemma-lb-intree}
Let~$d$ and~$p$ be non-negative integers and let~$S$ be a subtree of~$B_d$ with
    root~$s$ such that $t_d(V(S))\ge (2p+1)t_d(s)$.  Then~$S$ contains a minor
    of~$B_p$ rooted in~$s$.
\end{lemma}
\begin{proof}
We prove the statement by induction on the non-negative integer~$p$. The
    case~$p=0$ being trivial, suppose that $p\ge1$.  For~$x\in V(S)$, let~$S_x$
    be the subtree of~$S$ induced by~$x$ and all its descendants.  We can
    assume that~$t_d(V(S_x))<(2p+1)t_d(x)$ for every~$x\in V(S)\setminus
    \{s\}$, as otherwise we can consider~$S_x$ instead of~$S$ and combine the
    obtained minor with the path from~$x$ to~$s$ in~$S$.  In particular, for a
    child~$x_1$ of~$s$ in~$S$ we have
    $t_d(S_{x_1})<(2p+1)t_d(x_1)=(p+1/2)t_d(s)$, and thus~$t_d(V(S)\setminus
    (V(S_{x_1})\cup \{s\}))>(p-1/2)t_d(s)$.  Consequently,~$s$ has another
    child~$x_2$ in~$S$ and~$t_d(S_{x_2})=t_d(V(S)\setminus (V(S_{x_1})\cup
    \{s\}))>(p-1/2)t_d(s)=(2p-1)t_d(x_2)$.  Symmetrically,
    $t_d(S_{x_1})>(2p-1)t_d(x_1)$.  By the induction hypothesis, each of~$S_{x_1}$
    and~$S_{x_2}$ contains a minor of~$B_{p-1}$ rooted in~$x_1$ and~$x_2$,
    respectively, which combine with~$s$ to form a minor of~$B_p$ rooted in~$s$.
\end{proof}

Next, let us lift this result to~$T_d(B_d)$.  For a non-negative integer~$p$,
let us define~$q(p)\coloneqq10\sqrt{p+1}$.  Let us remark that the function $q$ is chosen so that
$q(p)-2\ge q(p-\lfloor (q(p)-4)/8\rfloor)$ holds for $p\ge 1$.
\begin{lemma}\label{lemma-lb-intt}
Let~$d$ and~$p$ be non-negative integers such that~$d\ge q(p)-1$,
    let~$G\coloneqq T_d(B_d)$ and~$w\coloneqq(t_d)_d$, let~$s$ be a vertex of~$G$ and let~$S$ be
    a connected induced subgraph of~$G$ contained in the jug of~$s$ and
    containing~$s$.  If~$w(V(S))\ge q(p)w(s)$, then~$S$ contains a minor
    of~$B_p$ rooted in~$s$.
\end{lemma}
\begin{proof}
We prove the statement by induction on the non-negative integer~$p$. The
    case~$p=0$ being trivial, suppose that~$p\ge 1$.  For~$x\in V(S)$,
    let~$S_x$ be the intersection of~$S$ with the jug of~$x$.  We can
    assume that $w(V(S_x))<q(p)w(x)$, as otherwise we can consider $S_x$
    instead of~$S$, find the required minor in~$S_x$, and combine it with a
    path from~$s$ to~$x$.

Let~$T$ be the subgraph of~$G$ induced by the neighbors of~$s$ in the jug
    of~$s$ (note that~$T$ is a copy of~$B_d$), and let~$N$ be the
    set of neighbors of~$s$ in~$G$ that belong to~$S$. Notice that~$N\subseteq V(T)$
    since~$S$ is contained in the jug of~$s$ by assumptions.
    We have
    \[q(p)w(s)\le w(V(S))=w(s)+\sum_{x\in N} w(V(S_x))<w(s)+q(p)w(N),\]
    and thus
\begin{equation}\label{eq-N}
w(N)>\Bigl(1-\frac{1}{q(p)}\Bigr)w(s)=\Bigl(1-\frac{1}{q(p)}\Bigr)w(V(T)).
\end{equation}
Let~$B$ consist of the vertices~$x$ in~$N$ such that
    $w(V(S_x))<(q(p)-2)w(x)$.  Then
    \begin{align*}
        q(p)w(s)&\le w(s)+\sum_{x\in N} w(V(S_x))\\
        &=w(s)+\sum_{x\in N\setminus B} w(V(S_x))+\sum_{x\in B} w(V(S_x))\\
        &<w(s)+q(p)w(N)-2w(B)\\
        &\le w(s)+q(p)w(s)-2w(B),
    \end{align*}
        and hence
\begin{equation}\label{eq-B}
w(B)< \frac{w(s)}{2}=\frac{w(V(T))}{2}.
\end{equation}
Let~$r$ be the root of~$T$ and set~$X\coloneqq(V(T)\setminus N)\cup\{r\}$. By~\eqref{eq-N} and the assumption $d\ge q(p)-1$,
\begin{equation}\label{eq-X}
w(X)<\frac{w(V(T))}{q(p)}+w(r)=\Bigl(\frac{1}{q(p)}+\frac{1}{d+1}\Bigr)w(V(T))\le \frac{2}{q(p)}w(V(T)).
\end{equation}
For~$x\in X$, let~$T_x$ be the subtree of the forest~$T[N\cup\{x\}]$ induced by~$x$ and its descendants, and set
$a(x)\coloneqq (w(V(T_x))-w(V(T_x)\cap B))/w(x)$.  By~\eqref{eq-B} and~\eqref{eq-X}, we have
\begin{equation}\label{eq-n4}
\frac{\sum_{x\in X}w(x)a(x)}{w(X)}=\frac{w(V(T))-w(B)}{w(X)}>\frac{q(p)}{4}.
\end{equation}
Since the left side of~\eqref{eq-n4} is a weighted average of the values~$a(x)$
    for~$x\in X$, there exists~$x\in X$ such that~$a(x)>q(p)/4$.  Let~$T'_x$ be
    the smallest subtree of~$T_x$ containing~$x$ and all vertices
    in~$V(T_x)\setminus B$.  Note that $w(V(T'_x))\ge
    a(x)w(x)>\frac{q(p)}{4}w(x)$ and no leaf of~$T'_x$ belongs to~$B$.
    Set~$p'\coloneqq\lfloor (q(p)-4)/8\rfloor$. Lemma~\ref{lemma-lb-intree}
    ensures that~$T'_x$ contains a minor~$\mu$ of~$T_{p'}$, which can be
    extended so that for every leaf~$u$, the bag~$\mu(u)$ contains a leaf~$y$
    of~$T'_x$; this leaf in particular does not belong to~$B$.
    Hence,~$w(V(S_y))\ge (q(p)-2)w(y)$, which is at least~$q(p-p')w(y)$ by the definition of $q$,
    since~$p\ge1$.  Consequently, the induction hypothesis implies that~$S_y$
    contains a minor of~$T_{p-p'}$ rooted in~$y$.  Adding these minors
    of~$T_{p-p'}$ for each leaf of~$T_{p'}$, and replacing $x$ by $s$ in the root bag,
    we obtain a minor of~$T_p$ in~$S$ rooted in $s$, as required.
\end{proof}

\noindent
We now use Lemma~\ref{lemma-lbgen} to give the desired lower bound.

\begin{theorem}\label{thm-lb-td}
Let~$r\colon \mathbb{N}\to\mathbb{R}_0^+$ be a non-decreasing function.
If all planar graphs of treewidth at most two are fractionally $\td$-fragile at rate~$r$, then
$r(a)=\Omega(a^2)$.
\end{theorem}
\begin{proof}
Consider two integers~$a$ and~$d$ such that~$d\ge a\ge 20$.  Let~$G\coloneqq T_d(T_d)$ and~$w\coloneqq (t_d)_d$.
Note that~$G$ is planar and has treewidth at most two.
Let~$X$ be a subset of~$V(G)$ such that $w(X)\le w(V(G))/a$.
Let~$r$ be the handle of~$G$ and let~$X'=X\cup \{r\}$; we have $w(X')\le w(X)+1=w(X)+\tfrac{w(V(G))}{d+1}\le\tfrac{2}{a}w(V(G))$.
For~$x\in X'$, let~$J_x$ be the jug of~$x$ and let~$S_x$ be the component of~$J_x-(X'\setminus\{x\})$ containing~$x$.
We have
\[\frac{a}{2}w(X')\le w(V(G))=\sum_{x\in X'} w(V(S_x)),\]
and thus there exists~$x\in X'$ such that $w(V(S_x))\ge \frac{a}{2}w(x)$.
    Set~$p\coloneqq\lfloor a^2/400-1\rfloor$, so~$p$ is a non-negative integer. Because~$10\sqrt{p+1}\le\frac{a}{2}$,
    we deduce from Lemma~\ref{lemma-lb-intt} that~$S_x$ contains a minor of~$T_p$.

Note that $T_p$ has treedepth $p+1$, that deleting a vertex decreases the treedepth by at most one,
and that treedepth is minor-monotone~\cite{nesbook}.  Since~$S_x-x\subseteq G-X$, we have
\[\td(G-X)\ge \td(S_x-x)\ge \td(S_x)-1\ge \td(T_p)-1\ge p.\]
Since this holds for every set~$X$ with~$w(X)\le w(V(G))/a$, Lemma~\ref{lemma-lbgen} implies that~$r(a)\ge p=\Omega(a^2)$.
\end{proof}

Outerplanar graphs are planar and have treewidth two; however, the graphs~$T_d(T_d)$ are not outerplanar if~$d\ge 2$.
As we will see below, outerplanar graphs are actually fractionally $\td$-fragile at a subquadratic rate.
Nevertheless, even for outerplanar graphs the rate is not linear, as we now
show.  Let us start by showing that~$T_d(P_n)$ has substantial treedepth, where~$P_n$ is the $n$-vertex path.
\begin{lemma}\label{lemma-tdpn}
Let~$d\ge 0$, $a\ge 1$ and~$n\ge 2^a$ be integers.  The graph~$T_d(P_n)$ has treedepth at least~$ad+1$.
\end{lemma}
\begin{proof}
We prove the statement by induction on the non-negative integer~$d$.  The case~$d=0$ is trivial, and we thus assume that~$d\ge 1$.
Set~$G\coloneqq T_d(P_n)$, let~$v$ be the handle of~$G$, and let~$Q$ be the $n$-vertex path induced by the neighbors of~$v$.
For a subpath~$Q'$ of~$Q$, we define~$J_{Q'}$ to be the union of the vertex sets of every jug the handle of which is contained in~$Q'$.
Suppose that~$R$ is a rooted tree witnessing the treedepth of~$T_d(P_n)$.  By finite induction we build a path~$u_0\dotso u_a$ in~$R$ starting at
    the root~$u_0$ of~$R$ and a decreasing sequence~$Q_0\supset Q_1\supset\dotsb\supset Q_a$ of subpaths of~$Q$ such that the following invariants hold for each~$i\in\{0,\dotsc, a\}$.
\begin{itemize}
\item[(i)] $|V(Q_i)|=2^{a-i}$;
\item[(ii)] $u_0,\dotsc, u_{i-1}\not\in J_{Q_i}$; and
\item[(iii)] the subtree of~$R$ rooted at~$u_i$ contains all vertices of~$J_{Q_i}$.
\end{itemize}
    We proceed by finite induction on~$i\in\{0,\dotsc,a\}$.
The path~$Q_0$ is chosen arbitrarily among the subpaths of~$Q$ with~$2^a$ vertices
    and~$u_0$ is the root of~$R$.
    For~$i\in\{1,\dotsc, a\}$, the path~$Q_i$ is selected as one of the halves of~$Q_{i-1}$ so that $J_{Q_i}$ does not contain~$u_{i-1}$.
    It follows that~$Q_i$ satisfies~(i) and~(ii).  By~(iii), we know that the
    subtree of~$R$ rooted~at $u_{i-1}$
contains all vertices of~$J_{Q_{i-1}}$, and thus also all vertices of~$J_{Q_i}$.
Since~$G[J_{Q_i}]$ is connected and~$u_{i-1}\notin J_{Q_i}$,
we can choose~$u_i$ as the unique child of~$u_{i-1}$ in~$R$ such that the
    subtree of~$R$ rooted at~$u_i$ contains all vertices of~$J_{Q_i}$, so
    that~(iii) is satisfied. This concludes the construction.

Now let~$x\in V(Q_a)$, let~$J_x$ be the jug of~$x$ and
    let~$R_x$ be the subtree of~$R$ rooted at~$u_a$.  We know by~(iii) that
    $V(J_x)\subseteq V(R_x)$.  Because~$J_x$ is isomorphic to~$T_{d-1}(P_n)$,
    the induction hypothesis implies that~$R_x$ has depth at
    least~$a(d-1)$.  Since~$u_a$ has depth~$a$ in~$R$, it follows that~$R$ has
    depth at least~$ad$.  Consequently, the treedepth of~$T_d(P_n)$ is at
    least~$ad+1$.
\end{proof}

We now give an argument analogous to that of Lemma~\ref{lemma-lb-intt}.
For the $d$-vertex path~$P_d$, let~$p_d\colon V(P_d)\to \mathbb{R}_0^+$
be the mapping that assigns~$1$ to each vertex of~$P_d$.
\begin{lemma}\label{lemma-lb-intpath}
Let~$d$,~$p$ and~$b$ be non-negative integers such that $d\ge 4b+2p+2$, let~$G\coloneqq T_d(P_d)$ and~$w\coloneqq (p_d)_d$. Let~$s$ be a vertex of~$G$
and let~$S$ be a connected induced subgraph of~$G$ contained in the jug of~$s$ and containing~$s$.
If $w(V(S))\ge (4b+2p+2)w(s)$, then~$S$ contains a minor of~$T_p(P_b)$ rooted in~$s$.
\end{lemma}
\begin{proof}
We prove the statement by induction on the non-negative integer~$p$. The case~$p=0$ being trivial, we suppose that $p\ge 1$.
For~$x\in V(S)$, let~$S_x$ be the intersection of~$S$ with the jug of~$x$.
We can assume that $w(V(S_x))<(4b+2p+2)w(x)$, as otherwise we can consider~$S_x$
instead of~$S$, find the required minor in~$S_x$, and combine it with a path from~$s$ to~$x$.

Let~$P$ be the subgraph of~$G$ induced by the neighbors of~$s$ in the jug
    of~$s$. Note that~$P$ is a copy~$v_1\dotso v_d$ of~$P_d$, and let~$N$ be the set of
    neighbors of~$s$ in~$G$ that are contained in~$S$. Notice that~$N\subseteq
    V(P)$ by hypothesis.  We have
    \begin{align*}(4b+2p+2)w(s)&\le w(V(S))=w(s)+\sum_{x\in N}w(V(S_x))\\
    &<w(s)+(4b+2p+2)w(N),\end{align*}
and thus
\begin{equation}\label{eq-Np}
w(N)>\left(1-\frac{1}{4b+2p+2}\right)w(s)=\left(1-\frac{1}{4b+2p+2}\right)w(V(P)).
\end{equation}
Let~$B$ consist of the vertices~$x$ in~$N$ such that~$w(V(S_x))<(4b+2p)w(x)$.
Since
    \begin{align*}
        (4b+2p+2)w(s)&\le w(V(S))=w(s)+\sum_{x\in N} w(V(S_x))\\
        &<w(s)+(4b+2p+2)w(N)-2w(B)\\
        &\le w(s)+(4b+2p+2)w(s)-2w(B),
    \end{align*}
we have
\begin{equation}\label{eq-Bp}
w(B)<\frac{w(s)}{2}=\frac{w(V(P))}{2}.
\end{equation}
Set~$X\coloneqq (V(P)\setminus N)\cup\{v_1\}$. By~\eqref{eq-Np}, we have
\begin{align}\label{eq-Xp}
    w(X)&<\frac{w(V(P))}{4b+2p+2}+w(v_1)\notag\\
    &=\left(\frac{1}{4b+2p+2}+\frac{1}{d}\right)w(V(P))\notag\\
    &\le \frac{1}{2b+p+1}w(V(P)).
\end{align}
    Given~$v_i,v_j\in V(P)$, the vertex~$v_j$ is \emph{to the right} of~$v_i$
    if~$j>i$.  For~$x\in X$, let~$P_x$ be the subpath of~$P[N\cup\{x\}]$
    induced by~$x$ and the vertices to the right of~$x$.
    Observe that~$(V(P_x))_{x\in X}$ is a partition of~$V(P)$.
    Consequently, setting
    $a(x)\coloneqq(w(V(P_x))-w(V(P_x)\cap B))/w(x)$, we deduce from~\eqref{eq-Bp}
    and~\eqref{eq-Xp} that
\begin{equation}\label{eq-wavg2}
        \frac{\sum_{x\in X}w(x)a(x)}{w(X)}=\frac{w(V(P))-w(B)}{w(X)}>\frac{2b+p+1}{2}.
\end{equation}
Since the left side of~\eqref{eq-wavg2} is a weighted average of the
    values~$a(x)$ for~$x\in X$, there exists $x\in X$ such that
    $a(x)>(2b+p+1)/2$.  Since all vertices of~$P$ have the same weight, we
    deduce that~$P_x-x$ contains at least~$(2b+p+1)/2-1\ge b$ vertices not
    belonging to~$B$.  For each such vertex~$y$, we have~$w(V(S_y))\ge
    (4b+2p)w(y)$, and by the induction hypothesis,~$S_y$ contains a minor
    of~$T_{p-1}(P_b)$ rooted in~$y$.  Since~$P_x-x$ is a subpath of~$P[N]$
    and~$s$ is adjacent to every vertex in~$N$, these minors along with~$s$
    combine to form a minor of~$T_p(P_b)$ rooted in~$s$ and contained in~$S$,
as required.
\end{proof}

\noindent
Now are ready to give a lower bound for outerplanar graphs.

\begin{theorem}\label{thm-lb-op-td}
Let~$r\colon \mathbb{N}\to\mathbb{R}_0^+$ be a non-decreasing function.
If all outerplanar graphs are fractionally $\td$-fragile at rate~$r$, then $r(a)=\Omega(a\log a)$.
\end{theorem}
\begin{proof}
Consider two integer~$a$ and~$d$ such that~$d\ge a\ge 28$.  Set~$G\coloneqq
    T_d(P_d)$ and~$w\coloneqq (p_d)_d$.  Note that~$G$ is outerplanar.  Let~$X$
    be a subset of~$V(G)$ such that~$w(X)\le w(V(G))/a$.  Let~$r$ be the handle
    of~$G$ and let~$X'\coloneqq X\cup \{r\}$; we have $w(X')\le
    w(X)+1=w(X)+\tfrac{w(V(G))}{d+1}\le\tfrac{2}{a}w(V(G))$.  For~$x\in X'$,
    let~$J_x$ be the jug of~$x$ and let~$S_x$ be the component
    of~$J_x-(X'\setminus\{x\})$ containing~$x$.  We have
\[\frac{a}{2}w(X')\le w(V(G))=\sum_{x\in X'} w(V(S_x)),\]
    and thus there exists~$x\in X'$ such that $w(V(S_x))\ge \frac{a}{2}w(X')$.
    Set~$c\coloneqq \lfloor (a-4)/12\rfloor$. We deduce from
    Lemma~\ref{lemma-lb-intpath} that~$S_x$ contains a minor of~$T_c(P_c)$.
    Since $S_x-x\subseteq G-X$, Lemma~\ref{lemma-tdpn} implies that
\[\td(G-X)\ge \td(S_x)-1\ge \td(T_c(P_c))-1\ge c\lfloor \log_2 c\rfloor.\]
    As this holds for every set~$X$ with~$w(X)\le w(V(G))/a$,
    Lemma~\ref{lemma-lbgen} implies that $r(a)\ge c\lfloor \log_2
    c\rfloor=\Omega(a\log a)$.
\end{proof}

Next, we will give a general upper bound for graphs with bounded treewidth.
To this end, we need the following property of treedepth.
\begin{lemma}\label{lemma-add}
Let~$H, H_1, \dotsc, H_t$ be induced subgraphs of a graph~$G$ such that
    \begin{itemize}
        \item $G=H\cup H_1\cup \dotsc\cup H_t$;
        \item $H_i\cap H_j\subseteq H$ whenever~$1\le i<j\le t$; and
        \item $H_i\cap H$ is a clique for~$i\in\{1,\dotsc,t\}$.
    \end{itemize}
Then
    \[\td(G)\le \td(H)+\max\{\td(H_i-V(H))\,:\,1\le i\le t\}.\]
\end{lemma}
\begin{proof}
Let~$T, T_1, \dotsc, T_t$ be rooted trees respectively witnessing the
    treedepths of~$H, H_1-V(H), \dotsc, H_t-V(H)$.  For~$i\in\{1,\dotsc,t\}$,
    since~$H\cap H_i$ is a clique, all its vertices are contained in a
    root-leaf path of~$T$; let~$\ell_i$ be the leaf of such a path.
    Taking~$T\cup T_1\cup \dotsb \cup T_t$ and, for~$i\in\{1,\dotsc, t\}$,
    adding an edge from the root of~$T_i$ to~$\ell_i$, we obtain a tree
    witnessing that the treedepth of~$G$ is at
    most~$\td(H)+\max\{\td(H_i-V(H))\,:\,1\le i\le t\}$.  \end{proof}

\noindent
We are now ready to give the following upper bound on the rate of
$\td$-fragility for graphs with bounded treewidth.

\begin{theorem}\label{thm-ub-tw-td}
For every non-negative integer~$t$, the class of graphs with treewidth at most~$t$ is fractionally $\td$-fragile
    at rate~$r(a)=2^{t(t+1)/2+1}a^t$.
\end{theorem}
\begin{proof}
We proceed by induction on the non-negative integer~$t$.  Graphs of
    treewidth~$0$ have no edges, and thus they have treedepth~$1$.  Hence,
    suppose that~$t\ge 1$.  Let~$a$ be a positive integer and let~$G$ be a
    graph of treewidth at most~$t$, which we can assume to be connected and
    chordal by Observation~\ref{obs-twcho} without loss of generality.  Let us
    fix an elimination ordering of~$G$, and let~$v$ be the first vertex in this
    ordering.  For a non-negative integer~$d$, let~$L_d$ be the set of vertices
    of~$G$ at distance exactly~$d$ from~$v$.  For~$i\in\{0,\dotsc, 2a-1\}$,
    let~$X_i\coloneqq \bigcup_{s\ge0}L_{i+s\cdot 2a}$,
    and choose~$X$ from the uniform distribution on~$\{X_0,\dotsc, X_{2a-1}\}$.

Consider a non-negative index~$j$, and note that~$G[L_j]$ has treewidth at
    most~$t-1$: for~$j=0$ it is obvious, while for~$j\ge 1$ it follows
    from the fact that each vertex of~$L_j$ has a neighbor in~$L_{j-1}$ preceding it in the elimination ordering by Observation~\ref{obs-prec}, and
    hence in the restriction of the elimination ordering to~$G[L_j]$, each
    vertex is preceded by at most~$t-1$ of its neighbors.  The
    induction hypothesis thus implies that for each~$j$, we can choose a set $Y_j\in
    \des{G[L_j]}{\td}{2^{(t-1)t/2+1}(2a)^{t-1}}$ at random such that $\Prb[v\in
    Y_j]\le \tfrac{1}{2a}$ for each~$v\in L_j$.  Set~$Z\coloneqq X\cup Y_0\cup Y_1\cup
    \dotsb$. If~$v\in V(G)$, then there exists a unique index~$j$ such that~$v\in L_j$, and hence
    \[\Prb[v\in Z]\le \Prb[v\in X]+\Prb[v\in Y_j]\le 1/a.\]
    Consequently, it suffices to show that $Z\in
    \des{G}{\td}{2^{t(t+1)/2+1}a^t}$.  As~$G[L_{j}\setminus Y_j]$
    has treedepth at most $2^{(t-1)t/2+1}(2a)^{t-1}=2^{t(t+1)/2}a^{t-1}$ for each~$j$, the conclusion follows by
    repeatedly applying Lemmas~\ref{lemma-compose} and~\ref{lemma-add}.
    Indeed, let~$i\in\{0,\dotsc,2a-1\}$ such that~$L_i\subseteq X$. It suffices to bound the treedepth of
    the subgraph of~$G$ induced by~$(L_{i+1}\setminus Y_{i+1})\cup\dotsb\cup(L_{i+2a-1}\setminus Y_{i+2a-1})$,
    and of that induced by~$\cup_{j=0}^{i-1}L_j\setminus Y_j$ in the border case --- which is omitted as
    similar to what follows only with different index boundaries yielding fewer applications of the lemmas.
    To this end, for any~$j\in\{1,\dotsc,2a-2\}$ define~$H$ to be~$G[\cup_{s=i+1}^{i+j}L_s\setminus Y_s]$
    and, for each component~$C$ of~$G[L_{i+j+1}\setminus Y_{i+j+1}]$, define~$H_C$ to be the subgraph of~$G$
    induced by the union of~$V(C)$ and the subset of vertices of~$H$ with a neighbor (in~$G$) that belongs to~$V(C)$.
    Lemma~\ref{lemma-compose} ensures that~$H\cap H_C$ is a clique,
    and Lemma~\ref{lemma-add} that the treedepth of~$H\cup\bigcup_{C}H_C$ is at most~$\td(H)+2^{t(t+1)/2}a^{t-1}$.
    Therefore, the conclusion follows by finite induction on~$j\in\{1,\dots,2a-1\}$.
\end{proof}

As we mentioned before, the bound provided by Theorem~\ref{thm-ub-tw-td} can be
improved for the special case of outerplanar graphs.  Firstly, we note that the
following holds.
\begin{observation}\label{obs-2cop}
Suppose that~$G$ is an outerplanar graph, that~$K\subseteq V(G)$ induces a connected subgraph of~$G$, and let~$H$
be a connected subgraph of~$G-K$ such that each vertex of~$H$ has a neighbor in~$K$.
Then~$H$ is a path.
\end{observation}
\begin{proof}
If~$H$ is not a path, it either is a cycle or contains a vertex of degree at
    least three. Contracting~$K$ to a single vertex, and considering it along
    with~$H$, we obtain either~$K_4$ or~$K_{2,3}$ as a minor of~$G$,
    contradicting the assumption that~$G$ is outerplanar.
\end{proof}

\noindent
We can now modify the argument used to demonstrate Theorem~\ref{thm-ub-tw-td}.
We use the fact that a path with $n$ vertices has treedepth $\lceil \log_2 (n+1)\rceil$,
see~\cite{nesbook}.
\begin{theorem}\label{thm-ub-op}
    The class of outerplanar graphs is fractionally $\td$-fragile at rate~$r(a)=2a(1+\lceil\log_2 a\rceil)$.
\end{theorem}
\begin{proof}
Let~$a$ be a positive integer.  Let~$G$ be an outerplanar graph, which without
    loss of generality can be assumed to be connected.  By triangulating the
    inner faces, we can assume that~$G$ is chordal.  Let~$L_0, L_1, \dotsc$
    and~$X_0, \dotsc, X_{2a-1}$, and~$X$  be defined in the same way as in the
    proof of Theorem~\ref{thm-ub-tw-td}.  From Lemma~\ref{lemma-compose} and
    Observation~\ref{obs-2cop}, we infer that~$G[L_j]$ is a disjoint union of
    paths, for each non-negative integer~$j$.  Repeating the same layering
    argument in~$G[L_j]$, for each~$j$, we can choose a set~$Y_j\subseteq L_j$
    at random so that $\Prb[v\in Y_j]\le \tfrac{1}{2a}$ for every~$v\in V(L_j)$
    and~$G[L_j\setminus Y_j]$ is a disjoint union of paths with at most~$2a-1$
    vertices.  Consequently, $\td(G[L_j\setminus Y_j])\le
    \lceil\log_2(2a)\rceil$ for every~$j$, and letting~$Z\coloneqq X\cup
    Y_0\cup Y_1\cup \dotsb$, we apply Lemmas~\ref{lemma-compose}
    and~\ref{lemma-add} similarly as in the proof of Theorem~\ref{thm-ub-tw-td}
    to infer that~$Z\in \des{G}{\td}{2a(1+\lceil\log_2 a\rceil)}$,
    while~$\Prb[v\in Z]\le 1/a$ for each~$v\in V(G)$.
\end{proof}

Combining Theorem~\ref{thm-ub-tw-td} with Corollary~\ref{cor-tw-frag}
yields that planar graphs are fractionally $\td$-fragile at rate~$a^{O(a)}$.
A much better bound can be obtained using Theorem~\ref{thm-trigeod}.
To this end, let us introduce another variation on Theorem~\ref{thm-ub-tw-td}.
\begin{theorem}\label{thm-ub-plch}
    The class of planar chordal graphs is fractionally $\td$-fragile at rate~$r(a)=8a^2(2+\lceil\log_2 a\rceil)$.
\end{theorem}
\begin{proof}
Let~$a$ be a positive integer.  Let~$G$ be a planar chordal graph, which
    without loss of generality can be assumed to be connected.  Let~$L_0,L_1
    \dotso$ and~$X_0, \dotsc, X_{2a-1}$, and~$X$  be defined in the same way as
    in the proof of Theorem~\ref{thm-ub-tw-td}.  For every $j\ge 1$, Lemma~\ref{lemma-compose}
    implies the neighborhood of every component of~$G[L_j]$ in $L_{j-1}$ induces a connected subgraph of $G$,
    and since $G$ is planar, $G[L_j]$ is outerplanar.  By
    Theorem~\ref{thm-ub-op}, we can for each~$j$ choose a set~$Y_j\subseteq
    L_j$ at random so that $\Prb[v\in Y_j]\le \tfrac{1}{2a}$ for each~$v\in
    V(L_j)$ and~$G[L_j\setminus Y_j]$ has treedepth at most~$4a(2+\lceil\log_2
    a\rceil)$.  Consequently, letting~$Z\coloneqq X\cup Y_0\cup Y_1\cup
    \dotsb$, Lemmas~\ref{lemma-compose} and~\ref{lemma-add} imply that $Z\in
    \des{G}{\td}{8a^2(2+\lceil\log_2 a\rceil)}$ similarly as before,
    while~$\Prb[v\in Z]\le 1/a$ for each~$v\in V(G)$.
\end{proof}
Let us point out that the rate from Theorem~\ref{thm-ub-op} cannot be
substantially improved: the graphs~$T_d(T_d(P_d))$ are planar, chordal, and
combining the ideas of Theorems~\ref{thm-lb-td} and~\ref{thm-lb-op-td}, one can
show that they cannot be fractionally $\td$-fragile at rate better
than~$\Omega(a^2\log a)$.  We can now compose the results to obtain a bound for
planar graphs.

\begin{corollary}\label{cor-planar}
The class of planar graphs is fractionally $\td$-fragile at
rate $384a^3(3+\lceil\log_2 a\rceil)$.
\end{corollary}
\begin{proof}
Let~$a$ be a positive integer.  Let~$G$ be a planar graph, which without loss of
    generality can be assumed to be connected.  Let~$v$ be a vertex of~$G$, and
    for~$d\ge 0$, let~$L_d$ be the set of vertices of~$G$ at distance
    exactly~$d$ from~$v$ in~$G$.  For~$i\in\{0,\dotsc, 2a-1\}$,
    let~$X_i\coloneqq L_i\cup L_{i+2a}\cup \dotsb$, and choose~$X_i$ from the
    uniform distribution on $\{X_0,\dotsc, X_{2a-1}\}$.

Consider any component~$H$ of~$G-X$. This component contains only vertices at
    distance from~$v$ between~$i+2aj+1$ and~$i+2aj+2a-1$ for some
    integer~$j$.  Let~$G'$ be the graph obtained from~$G$ by deleting all
    vertices at distance at least~$i+2aj+2a$ from~$v$ and by contracting all
    vertices at distance at most~$\max(i+2aj,0)$ from~$v$ to a single
    vertex~$x$.  Clearly,~$G'$ is a minor of~$G$, and thus~$G'$ is planar.
    Moreover, every vertex of~$G'$ is at distance at most~$2a-1$ from~$x$
    and~$H\subseteq G'$.  Let~$G''$ be a triangulation of~$G'$, and let~$\PP$ be
    a trigeodesic partition of vertices of~$G''$ such that~$G''/\PP$ is
    chordal, which exists by Theorem~\ref{thm-trigeod}.  By
    Theorem~\ref{thm-ub-plch}, we can choose a set~$Y'_H\subseteq V(G''/\PP)$
    at random such that $\td(G''/\PP-Y'_H)\le
    32a^2(3+\log_2 a)$ and~$\Prb[z\in Y'_H]\le \tfrac{1}{2a}$ for every~$z\in
    V(G''/\PP)$.  We can naturally view~$Y'_H$ as a subset of~$\PP$; with this
    in mind, let~$Y_H\coloneqq V(H)\cap \bigcup_{P\in Y'_H} P$.  Clearly, for
    every~$u\in V(H)$, we have $\Prb[u\in Y_H]\le \tfrac{1}{2a}$.  Furthermore,
    note that since~$G''$ has radius less than~$2a$, every geodesic path
    in~$G''$ has less than~$4a$ vertices, and since $\PP$ is trigeodesic, if follows
    that~$|P|<12a$ for every~$P\in\PP$.  Consequently, we can turn the tree~$T$
    witnessing the treedepth of~$G''/\PP-Y'_H$ into one for~$H-Y_H$ by
    replacing each vertex $P\in V(G''/\PP-Y'_H)$ in~$T$ by a path consisting of
    the vertices contained in~$P\cap V(H)$.  Therefore
    $\td(H-Y_H)<12a\td(G''/\PP-Y'_H)\le 384a^3(3+\lceil\log_2 a\rceil)$.

Letting~$Z$ be the union of~$X_i$ and the sets~$Y_H$ for each component~$H$
    of~$G-X_i$, we conclude that $\td(G-Z)\le 384a^3(3+\lceil\log_2 a\rceil)$
    and~$\Prb[u\in Z]\le 1/a$ for each~$u\in V(G)$.
\end{proof}

As we mentioned before, the planar graphs~$T_d(T_d(P_d))$ cannot be
fractionally $\td$-fragile at rate better than~$\Omega(a^2\log a)$.  We leave
open the question of what is the correct rate for planar graphs (between the
bounds of~$\Omega(a^2\log a)$ and~$O(a^3\log a)$ we obtained).


\end{document}